\documentclass[12pt,a4paper]{article}
\usepackage[utf8x]{inputenc}
\usepackage{amsmath}
\usepackage{amsfonts}
\usepackage{amssymb}
\usepackage{amsthm}
\usepackage{mathrsfs}
\usepackage{fancyhdr}
\usepackage{graphicx}
\usepackage[usenames,x11names]{xcolor}
\usepackage{arabtex}
\usepackage{framed}
\usepackage[top=2cm,bottom=2cm,margin=2cm]{geometry}
\usepackage{cancel}
\usepackage{array}

\usepackage{hyperref}

\title{Nontrivial lower bounds for the $p$-adic valuations of some type of rational numbers}
\author{\sc Bakir FARHI \\
Laboratoire de Mathématiques appliquées \\
Faculté des Sciences Exactes \\
Université de Bejaia, 06000 Bejaia, Algeria \\[1mm]
\href{mailto:bakir.farhi@gmail.com}{bakir.farhi@gmail.com} \\[1mm]
\url{http://farhi.bakir.free.fr/}
}
\date{}

\let\up=\textsuperscript

\def\R{{\mathbb R}}
\def\Q{{\mathbb Q}}
\def\C{{\mathbb C}}
\def\N{{\mathbb N}}
\def\Z{{\mathbb Z}}
\def\lcm{\mathrm{lcm}}

\def\restmod#1#2{#1\ (\mathrm{mod}\ #2)}

\def\idem{\leavevmode\hbox to 10.6mm{\vrule height .63ex depth -.59ex
    width 10mm\hfill}}

\theoremstyle{plain}
\numberwithin{equation}{section}
\newtheorem{thm}{Theorem}[section]

\theoremstyle{definition}

\theoremstyle{remark}
\newtheorem{rmk}[thm]{Remark}

\begin{document}
\maketitle

\begin{abstract}
In this paper, we will show that the $p$-adic valuation (where $p$ is a given prime number) of some type of rational numbers is unusually large. This generalizes the very recent results by the author and by A. Dubickas, which are both related to the special case $p = 2$. The crucial point for obtaining our main result is the fact that the $p$-adic valuation of the rational numbers in question is unbounded from above. We will confirm this fact by three different methods; the first two are elementary while the third one leans on the $p$-adic analysis.        
\end{abstract}


\section{Introduction and Notation}\label{sec1}

Throughout this paper, we let $\N$ denote the set of positive integers and $\N_0 := \N \cup \{0\}$ denote the set of non-negative integers. For $x \in \R$, we let $\lfloor x\rfloor$ denote the integer part of $x$. For a given prime number $p$ and a given non-zero rational number $r$, we let $\vartheta_p(r)$ denote the usual $p$-adic valuation of $r$; if in addition $r$ is positive then we let $\log_p(r)$ denote its logarithm to the base $p$ (i.e., $\log_p(r) := \frac{\log{r}}{\log{p}}$). Next, the least common multiple of given positive integers $u_1 , u_2 , \dots , u_n$ ($n \in \N$) is denoted by $\lcm(u_1 , u_2 , \dots , u_n)$ or by $\lcm\{u_1 , u_2 , \dots , u_n\}$ if this is more convenient. In several places of this paper, we will use the immediate estimate $\vartheta_p(n) \leq \log_p(n)$ (for any prime $p$ and any $n \in \N$). We also often use the immediate formula $\vartheta_p\left(\lcm(1 , 2 , \dots , n)\right) = \left\lfloor\log_p(n)\right\rfloor$ (for any prime $p$ and any $n \in \N$). At the end of the paper, we need to use the $p$-adic logarithm function which we denote by $L_p$ (to differentiate from the notation $\log_p$, which is reserved to denote the logarithm to the base $p$). With the usual notation $\Q_p$ for the field of $p$-adic numbers, $\C_p$ for the field of the $p$-adic complex numbers, and ${|\cdot|}_p$ for the usual $p$-adic absolute value on $\C_p$, recall that $L_p$ can be defined by:
$$
- L_p(1 - x) := \sum_{k = 1}^{+ \infty} \frac{x^k}{k} ~~~~~~~~~~ (\forall x \in \C_p , {|x|}_p < 1) .
$$
(See \cite{kob}). The fundamental property of $L_p$ is that it satisfies the functional equation:
$$
L_p(u v) = L_p(u) + L_p(v)
$$
(for all $u , v \in \C_p$, with ${|u - 1|}_p < 1$ and ${|v - 1|}_p < 1$).

In \cite{far1,far2}, the author have obtained nontrivial lower bounds for the $2$-adic valuation of the rational numbers of the form $\sum_{k = 1}^{n} \frac{2^k}{k}$ ($n \in \N$). The stronger one is
\begin{equation}\label{eqint1}
\vartheta_2\left(\sum_{k = 1}^{n} \frac{2^k}{k}\right) \geq n - \left\lfloor\log_2(n)\right\rfloor ~~~~~~~~~~ (\forall n \in \N) .
\end{equation}
In \cite{far2}, the author also posed the problem of generalizing \eqref{eqint1} to other prime numbers $p$ other than $p = 2$. In \cite{dub}, Dubickas has found arguments to improving and optimizing \eqref{eqint1} by leaning only on the fact that the sequence $\left\{\vartheta_2\left(\sum_{k = 1}^{n} \frac{2^k}{k}\right)\right\}_{n \geq 1}$ is unbounded from above. However, he does not established any way to prove this fact without using \eqref{eqint1}. The main result in \cite{dub} states that we have for any $n \in \N$:
\begin{equation}\label{eqint2}
\vartheta_2\left(\sum_{k = 1}^{n} \frac{2^k}{k}\right) \geq (n + 1) - \log_2(n + 1) ,
\end{equation}
with equality if and only if $n$ has the form $n = 2^{\alpha} - 1$ ($\alpha \in \N$).

The goal of this paper is twofold. On the one hand, we expand and improve the arguments in \cite{dub} to establish a general result providing to us nontrivial lower bounds for the $p$-adic valuation of a sum of rational numbers under some conditions (see Theorem \ref{t2}). On the other hand, we solve the problem posed in \cite{far2} by generalizing \eqref{eqint1} and \eqref{eqint2} to other prime numbers. Precisely, we show (in different ways) that for any prime number $p$ and any non-multiple integer $a$ of $p$, the sequence
\begin{equation}\label{eqint3}
\left\{\vartheta_p\left(\sum_{k = 1}^{n} \left(\frac{1}{a^k} + \frac{1}{(p - a)^k}\right) \frac{p^k}{k}\right)\right\}_{n \geq 1}
\end{equation}
is unbounded from above. Then, by using our general theorem \ref{t2}, we derive an optimal lower bound for the sequence in \eqref{eqint3}. It must be noted that the crucial point of the non-boundness from above of the sequence in \eqref{eqint3} is established by three methods. The first two are elementary and effective while the third one leans on the $p$-adic analysis and it is ineffective; precisely, it uses the function $L_p$ described above. Personally, we consider that the deep reason why the sequence in \eqref{eqint3} is unbounded from above is rather given by the third method.    

\section{The results and the proofs}

Our main result is the following:

\begin{thm}\label{t1}
Let $p$ be a prime number and $a$ be an integer not multiple of $p$. Then we have for all positive integer $n$:
\begin{equation}\label{eq4}
\vartheta_p\left(\sum_{k = 1}^{n}\left(\frac{1}{a^k} + \frac{1}{(p - a)^k}\right) \frac{p^k}{k}\right) \geq (n + 1) - \log_p\left(\frac{n + 1}{2}\right) .
\end{equation}
In addition, this inequality becomes an equality if and only if $n$ has the form $n = 2 p^{\alpha} - 1$ {\rm(}$\alpha \in \N_0${\rm)}.
\end{thm}

Note that Theorem \ref{t1} generalizes the recent results of the author \cite{far1,far2} and Dubickas \cite{dub}, which are both related to the particular case $p = 2$. Especially, if we take $p = 2$ and $a = 1$ in Theorem \ref{t1}, we exactly obtain (after some obvious simplifications) the main result of \cite{dub}, stating that:
$$
\vartheta_2\left(\sum_{k = 1}^{n} \frac{2^k}{k}\right) \geq (n + 1) - \log_2(n + 1) ~~~~~~~~~~ (\forall n \in \N) ,
$$
with equality if and only if $n$ has the form $(2^{\alpha} - 1)$ ($\alpha \in \N$).

The proof of Theorem \ref{t1} is based in part on the following result which can serve us well in other situations for bounding from below the $p$-adic valuation of a sum of rational numbers when it is unbounded from above. It must be also noted that the result below is obtained by generalizing the arguments in \cite{dub}. 

\begin{thm}\label{t2}
Let $p$ be a fixed prime number and ${(r_n)}_{n \geq 1}$ be a sequence of rational numbers such that the sequence $\left\{\vartheta_p\left(\sum_{k = 1}^{n} r_k\right)\right\}_{n \geq 1}$ is unbounded from above. Let also ${(\ell_k)}_{k \geq 2}$ be an increasing real sequence satisfying the property:
\begin{equation}\label{eq1}
\ell_k \leq \vartheta_p(r_k) ~~~~~~~~~~ (\forall k \geq 2) .
\end{equation}
Then we have for any positive integer $n$:
\begin{equation}\label{eq2}
\vartheta_p\left(\sum_{k = 1}^{n} r_k\right) \geq \min_{k \geq n + 1} \vartheta_p(r_k) \geq \ell_{n + 1} .
\end{equation}
In addition, the inequality $\vartheta_p\left(\sum_{k = 1}^{n} r_k\right) \geq \ell_{n + 1}$ becomes an equality if and only if we have
\begin{equation}\label{eq3}
\min_{k \geq n + 1} \vartheta_p(r_k) = \ell_{n + 1} .
\end{equation}
\end{thm}

Our main result (i.e., Theorem \ref{t1}) is proved in two steeps: in the first one, we suppose (in the situation of Theorem \ref{t1}) that the sequence $\left\{\vartheta_p\left(\sum_{k = 1}^{n} \left(\frac{1}{a^k} + \frac{1}{(p - a)^k}\right) \frac{p^k}{k}\right)\right\}_{n \geq 1}$ is unbounded from above and we apply for it Theorem \ref{t2} to prove the lower bound \eqref{eq4} and to characterize the $n$'s for which it is attained. In the second one, we return to prove the non-boundness from above of the considered sequence. We do this by three different methods: the first one is based on two identities, one is combinatorial and the other is arithmetic. The second one uses a certain functional equation and the Taylor polynomials. The third one uses the $p$-adic analysis; precisely the $p$-adic logarithm function.

Let us begin by proving Theorem \ref{t2}.

\begin{proof}[Proof of Theorem \ref{t2}]
Let $n$ be a fixed positive integer. Let us show the first inequality of \eqref{eq2}. Since, by hypothesis, the sequence $\left\{\vartheta_p\left(\sum_{k = 1}^{N} r_k\right)\right\}_{N \geq 1}$ is unbounded from above then there exists $m \in \N$, with $m > n$, such that:
$$
\vartheta_p\left(\sum_{k = 1}^{m} r_k\right) > \vartheta_p\left(\sum_{k = 1}^{n} r_k\right) .
$$
Then, by using the elementary properties of the $p$-adic valuation, we have on the one hand:
$$
\vartheta_p\left(\sum_{k = n + 1}^{m} r_k\right) = \vartheta_p\left(\sum_{k = 1}^{m} r_k - \sum_{k = 1}^{n} r_k\right) = \min\left(\vartheta_p\left(\sum_{k = 1}^{m} r_k\right) , \vartheta_p\left(\sum_{k = 1}^{n} r_k\right)\right) = \vartheta_p\left(\sum_{k = 1}^{n} r_k\right) 
$$
and on the other hand:
$$
\vartheta_p\left(\sum_{k = n + 1}^{m} r_k\right) \geq \min_{n + 1 \leq k \leq m} \vartheta_p(r_k) \geq \min_{k \geq n + 1} \vartheta_p(r_k) .
$$
By comparing these two results, we deduce that:
$$
\vartheta_p\left(\sum_{k = 1}^{n} r_k\right) \geq \min_{k \geq n + 1} \vartheta_p(r_k) ,
$$
which is nothing else the first inequality of \eqref{eq2}. The second inequality of \eqref{eq2} is immediately derived from its first one together with the properties of the sequence ${(\ell_k)}_{k \geq 2}$. Indeed, we have
\begin{align*}
\vartheta_p\left(\sum_{k = 1}^{n} r_k\right) & \geq \min_{k \geq n + 1} \vartheta_p(r_k) ~~~~~~~~~~ (\text{by the first inequality of \eqref{eq2}}) \\
& \geq \min_{k \geq n + 1} \ell_k ~~~~~~~~~~~~~~~~ (\text{by using \eqref{eq1}}) \\
& = \ell_{n + 1} ~~~~~~~~~~~~~~~~~~~~ (\text{since } {(\ell_k)}_k \text{ is increasing by hypothesis}) ,
\end{align*}
confirming the second inequality of \eqref{eq2}.

Now, let us prove the second part of Theorem \ref{t2}. If $\vartheta_p\left(\sum_{k = 1}^{n} r_k\right) = \ell_{n + 1}$ then we have (according to \eqref{eq2}, proved above): $\min_{k \geq n + 1} \vartheta_p(r_k) = \ell_{n + 1}$, as required. Conversely, suppose that $\min_{k \geq n + 1} \vartheta_p(r_k) = \ell_{n + 1}$ and let us show that $\vartheta_p\left(\sum_{k = 1}^{n} r_k\right) = \ell_{n + 1}$. To do so, we first show that:
\begin{equation}\label{eq5}
\vartheta_p(r_{n + 1}) < \vartheta_p(r_k) ~~~~~~~~~~ (\forall k > n + 1) .
\end{equation}
To prove \eqref{eq5}, let us argue by contradiction. So, suppose that there is an integer $k_0 > n + 1$ which satisfies $\vartheta_p(r_{n + 1}) \geq \vartheta_p(k_0)$. So we have
$$
\min_{k \geq n + 1} \vartheta_p(r_k) = \min_{k \geq n + 2} \vartheta_p(r_k) \geq \min_{k \geq n + 2} \ell_k = \ell_{n + 2} > \ell_{n + 1}
$$
(according to \eqref{eq1} and the increase of ${(\ell_k)}_k$), contradicting the supposition $\min_{k \geq n + 1} \vartheta_p(r_k) = \ell_{n + 1}$. This contradiction confirms \eqref{eq5}. Now, we shall use \eqref{eq5} to prove the desired equality $\vartheta_p\left(\sum_{k = 1}^{n} r_k\right) = \ell_{n + 1}$. On the one hand, we have (according to the first part of this proof):
$$
\vartheta_p\left(\sum_{k = n + 1}^{m} r_k\right) = \vartheta_p\left(\sum_{k = 1}^{n} r_k\right) .
$$
But on the other hand, we have (according to \eqref{eq5} and the elementary properties of the $p$-adic valuation):
$$
\vartheta_p\left(\sum_{k = n + 1}^{m} r_k\right) = \min_{n + 1 \leq k \leq m} \vartheta_p(r_k) = \vartheta_p(r_{n + 1}) = \min_{k \geq n + 1} \vartheta_p(r_k) = \ell_{n + 1} .
$$
Comparing the two results, we derive the required equality: $\vartheta_p\left(\sum_{k = 1}^{n} r_k\right) = \ell_{n + 1}$. This confirms the second part of Theorem \ref{t2} and completes this proof.
\end{proof}   

Next, we have the following fundamental result:

\begin{thm}\label{t3}
Let $p$ be a prime number and $a$ be an integer not multiple of $p$. Then the sequence
$$
\left\{\vartheta_p\left(\sum_{k = 1}^{n} \left(\frac{1}{a^k} + \frac{1}{(p - a)^k}\right) \frac{p^k}{k}\right)\right\}_{n \geq 1}
$$
is unbounded from above.
\end{thm}

Admitting for the moment Theorem \ref{t3}, our main result is obtained as an application of Theorem \ref{t2}.

\begin{proof}[Proof of Theorem \ref{t1} by admitting Theorem \ref{t3}]
Let us put ourselves in the situation of Theorem \ref{t1}. We apply Theorem \ref{t2} with $r_k := \left(\frac{1}{a^k} + \frac{1}{(p - a)^k}\right) \frac{p^k}{k}$ ($\forall k \in \N$) and $\ell_k := k - \log_p\left(\frac{k}{2}\right)$ ($\forall k \geq 2$). The non-boundness from above of the sequence $\left\{\vartheta_p\left(\sum_{k = 1}^{n} r_k\right)\right\}_{n \geq 1}$ is guaranteed by Theorem \ref{t3} (admitted for the moment). Next, the increase of the sequence ${(\ell_k)}_{k \geq 2}$ can be derived from the increase of the function $x \mapsto x - \log_p\left(\frac{x}{2}\right)$ in the interval $[2 , + \infty)$. Finally, we have for any integer $k \geq 2$: 
\begin{align*}
\vartheta_p(r_k) & = \vartheta_p\left(\left(\frac{1}{a^k} + \frac{1}{(p - a)^k}\right) \frac{p^k}{k}\right) \\[1mm]
& = \vartheta_p\left(\frac{a^k + (p - a)^k}{\left(a (p - a)\right)^k} \cdot \frac{p^k}{k}\right) \\[1mm]
& = \vartheta_p\left(a^k + (p - a)^k\right) + k - \vartheta_p(k) ~~~~~~~~~~ (\text{since } a \text{ is coprime with } p) .
\end{align*}
If $k$ is even, we use $\vartheta_p\left(a^k + (p - a)^k\right) \geq \vartheta_p(2)$ (for $p > 2$, this is obvious and for $p = 2$, observe that $a^k + (p - a)^k$ is even). So, we obtain
$$
\vartheta_p(r_k) \geq k - \vartheta_p\left(\frac{k}{2}\right) \geq k - \log_p\left(\frac{k}{2}\right) = \ell_k
$$
(because $k / 2$ is a positive integer if $k$ is even). However, if $k$ is odd, we use $\vartheta_p\left(a^k + (p - a)^k\right) \geq 1$ (since $a^k + (p - a)^k \equiv \restmod{a^k + (- a)^k}{p} \equiv \restmod{0}{p}$). So, we obtain again:
$$
\vartheta_p(r_k) \geq 1 + k - \vartheta_p(k) \geq \log_p(2) + k - \log_p(k) = k - \log_p\left(\frac{k}{2}\right) = \ell_k .
$$
Consequently, we have for any integer $k \geq 2$: $\vartheta_p(r_k) \geq \ell_k$. So, all the hypothesis of Theorem \ref{t2} are satisfied; thus we can apply it for our situation. Applying the first part of Theorem \ref{t2}, we get for any positive integer $n$:
$$
\vartheta_p\left(\sum_{k = 1}^{n} \left(\frac{1}{a^k} + \frac{1}{(p - a)^k}\right) \frac{p^k}{k}\right) \geq \min_{k \geq n + 1} \vartheta_p\left(\left(\frac{1}{a^k} + \frac{1}{(p - a)^k}\right) \frac{p^k}{k}\right) \geq (n + 1) - \log_p\left(\frac{n + 1}{2}\right) ,
$$
confirming Inequality \eqref{eq4} of Theorem \ref{t1}. Next, for a given positive integer $n$, the second part of Theorem \ref{t2} tells us that \eqref{eq4} becomes an equality if and only if we have
$$
\min_{k \geq n + 1} \vartheta_p\left(\left(\frac{1}{a^k} + \frac{1}{(p - a)^k}\right) \frac{p^k}{k}\right) = (n + 1) - \log_p\left(\frac{n + 1}{2}\right) ,
$$
that is
\begin{equation}\label{eq6}
\min_{k \geq n + 1} \vartheta_p\left(\left(a^k + (p - a)^k\right) \frac{p^k}{k}\right) = (n + 1) - \log_p\left(\frac{n + 1}{2}\right) .
\end{equation}
So, it remains to prove that \eqref{eq6} holds if and only if $n$ has the form $n = 2 p^{\alpha} - 1$ ($\alpha \in \N_0$). Let us prove this last fact. \\
\textbullet{} Suppose that \eqref{eq6} holds. Then, we have
$$
\log_p\left(\frac{n + 1}{2}\right) = (n + 1) - \min_{k \geq n + 1} \vartheta_p\left(\left(a^k + (p - a)^k\right) \frac{p^k}{k}\right) \in \Z .
$$
But since $\log_p\left(\frac{n + 1}{2}\right) \geq 0$, we have even $\log_p\left(\frac{n + 1}{2}\right) \in \N_0$. By setting $\alpha := \log_p\left(\frac{n + 1}{2}\right) \in \N_0$, we get $n = 2 p^{\alpha} - 1$, as required. \\
\textbullet{} Conversely, suppose that $n = 2 p^{\alpha} - 1$ for some $\alpha \in \N_0$. We first claim that we have
\begin{equation}\label{eq7}
\vartheta_p\left(a^{n + 1} + (p - a)^{n + 1}\right) = \vartheta_p(2) .
\end{equation}
To confirm \eqref{eq7}, we distinguish two cases: \\
--- \textsc{1\up{st} case:} (If $p = 2$). In this case, because $a$ is coprime with $p$ then $a$ and $(p - a)$ are both odd, implying that $a^2 \equiv \restmod{1}{4}$ and $(p - a)^2 \equiv \restmod{1}{4}$. Then, because $n + 1 = 2 p^{\alpha}$ is even, we have also $a^{n + 1} \equiv \restmod{1}{4}$ and $(p - a)^{n + 1} \equiv \restmod{1}{4}$; thus $a^{n + 1} + (p - a)^{n + 1} \equiv \restmod{2}{4}$, implying that $\vartheta_p\left(a^{n + 1} + (p - a)^{n + 1}\right) = 1 = \vartheta_p(2)$. \\
--- \textsc{2\up{nd} case:} (If $p$ is odd). In this case, because $n + 1 = 2 p^{\alpha}$ is even, we have $a^{n + 1} + (p - a)^{n + 1} \equiv \restmod{a^{n + 1} + (- a)^{n + 1}}{p} \equiv \restmod{2 a^{n + 1}}{p} \not\equiv \restmod{0}{p}$ (since $p$ is assumed odd and $a$ is coprime with $p$). Thus $\vartheta_p\left(a^{n + 1} + (p - a)^{n + 1}\right) = 0 = \vartheta_p(2)$. Our claim \eqref{eq7} is proved. Now, using \eqref{eq7}, we have
\begin{multline*}
\vartheta_p\left(\left(a^{n + 1} + (p - a)^{n + 1}\right) \frac{p^{n + 1}}{n + 1}\right) = \vartheta_p(2) + (n + 1) - \vartheta_p(n + 1) = (n + 1) - \vartheta_p\left(\frac{n + 1}{2}\right) \\
= n + 1 - \vartheta_p\left(p^{\alpha}\right) = n + 1 - \alpha = n + 1 - \log_p\left(\frac{n + 1}{2}\right) .
\end{multline*}
This shows that \eqref{eq6} is equivalent to:
$$
\vartheta_p\left(\left(a^k + (p - a)^k\right) \frac{p^k}{k}\right) \geq (n + 1) - \log_p\left(\frac{n + 1}{2}\right) ~~~~~~~~~~ (\forall k \geq n + 2) ,
$$
which is weaker than:
\begin{equation}\label{eq8}
k - (n + 1) \geq \vartheta_p(k) - \log_p\left(\frac{n + 1}{2}\right) ~~~~~~~~~~ (\forall k \geq n + 2) .
\end{equation}
Let us prove \eqref{eq8} for a given integer $k \geq n + 2$. To do so, we distinguish two cases according to whether $\vartheta_p(k) \leq \alpha + 1$ or not. \\
--- \textsc{1\up{st} case:} (If $\vartheta_p(k) \leq \alpha + 1$). In this case, we have
$$
k - (n + 1) \geq 1 ~~~~\text{and}~~~~ \vartheta_p(k) - \log_p\left(\frac{n + 1}{2}\right) = \vartheta_p(k) - \alpha \leq 1 .
$$
Thus \eqref{eq8} is true. \\
--- \textsc{2\up{nd} case:} (If $\vartheta_p(k) \geq \alpha + 2$). In this case, we have $k \geq p^{\alpha + 2} \geq 2 p^{\alpha + 1}$. Hence
\begin{multline*}
k - (n + 1) = k - 2 p^{\alpha} \geq \frac{p - 1}{p} k \geq \frac{k}{p} \geq p^{\vartheta_p(k) - 1} \geq 2^{\vartheta_p(k) - 1} \geq \vartheta_p(k) \geq \vartheta_p(k) - \log_p\left(\frac{n + 1}{2}\right) ,
\end{multline*}
confirming \eqref{eq8} for this case also.

Consequently, \eqref{eq8} is valid for any integer $k \geq n + 2$. This completes the proof of the second part of Theorem \ref{t1} and achieves this proof.
\end{proof}

The rest of the paper is now devoted to prove Theorem \ref{t3}. We achieve this by three different methods:

\subsection*{The first method}
We lean on two identities. The first one (due to Mansour \cite{man}) is combinatorial and states that:
\begin{equation}\label{eq9}
\sum_{k = 0}^{n} \frac{x^k y^{n - k}}{\binom{n}{k}} = \frac{n + 1}{(x + y) \left(\frac{1}{x} + \frac{1}{y}\right)^{n + 1}} \sum_{k = 1}^{n + 1} \frac{\left(x^k + y^k\right) \left(\frac{1}{x} + \frac{1}{y}\right)^k}{k}
\end{equation}
(for any $x , y \in \R^*$, with $x + y \neq 0$, and any $n \in \N_0$). While the second one (due to the author \cite{far3}) is arithmetic and states that:
\begin{equation}\label{eq10}
\lcm\left\{\binom{n}{0} , \binom{n}{1} , \dots , \binom{n}{n}\right\} = \frac{\lcm\left(1 , 2 , \dots , n , n + 1\right)}{n + 1}
\end{equation}
(for any $n \in \N_0$).

Using \eqref{eq9} and \eqref{eq10}, we are now ready to prove Theorem \ref{t3}. Let $p$ be a prime number and $a$ be an integer non-multiple of $p$. By applying \eqref{eq9} for $x = a$ and $y = p - a$ and replacing $n$ by $(n - 1)$ (where $n \in \N$), we get (after simplifying and rearranging)

\begin{equation}\label{eq11}
\sum_{k = 1}^{n} \left(\frac{1}{a^k} + \frac{1}{(p - a)^k}\right) \frac{p^k}{k} = \frac{p^{n + 1}}{n \left(a (p - a)\right)^n} \sum_{k = 0}^{n - 1} \frac{a^k (p - a)^{n - 1 - k}}{\binom{n - 1}{k}} .
\end{equation} 

On the other hand, for any $n \in \N$, we have (according to \eqref{eq10}):

\begin{equation}\label{eq12}
1 = \frac{n}{\lcm(1 , 2 , \dots , n)} \lcm\left\{\binom{n - 1}{0} , \binom{n - 1}{1} , \dots , \binom{n - 1}{n - 1}\right\} .
\end{equation}

Then, for a given $n \in \N$, by multiplying side to side \eqref{eq11} and \eqref{eq12}, we obtain
\begin{multline*}
\sum_{k = 1}^{n} \left(\frac{1}{a^k} + \frac{1}{(p - a)^k}\right) \frac{p^k}{k} = \frac{p^{n + 1}}{\left(a (p - a)\right)^n \lcm(1 , 2 , \dots , n)} \\
\times \lcm\left\{\binom{n - 1}{0} , \binom{n - 1}{1} , \dots , \binom{n - 1}{n - 1}\right\} \sum_{k = 0}^{n - 1} \frac{a^k (p - a)^{n - 1 - k}}{\binom{n - 1}{k}} .
\end{multline*}
But since the rational number
$$
\lcm\left\{\binom{n - 1}{0} , \binom{n - 1}{1} , \dots , \binom{n - 1}{n - 1}\right\} \sum_{k = 0}^{n - 1} \frac{a^k (p - a)^{n - 1 - k}}{\binom{n - 1}{k}} 
$$
is obviously an integer, we derive from the last identity that:
\begin{align*}
\vartheta_p\left(\sum_{k = 1}^{n} \left(\frac{1}{a^k} + \frac{1}{(p - a)^k}\right) \frac{p^k}{k}\right) & \geq \vartheta_p\left(\frac{p^{n + 1}}{\left(a (p - a)\right)^n \lcm(1 , 2 , \dots , n)}\right) \\
& \hspace*{-1cm} = n + 1 - \vartheta_p\left(\lcm(1 , 2 , \dots , n)\right) ~~~~~ (\text{since } a \text{ is not a multiple of } p) \\
& \hspace*{-1cm} = n + 1 - \left\lfloor\log_p(n)\right\rfloor ,
\end{align*}
confirming the non-boundness from above of the sequence $\left\{\vartheta_p\left(\sum_{k = 1}^{n} \left(\frac{1}{a^k} + \frac{1}{(p - a)^k}\right) \frac{p^k}{k}\right)\right\}_{n \geq 1}$ (since $n + 1 - \left\lfloor\log_p(n)\right\rfloor \rightarrow + \infty$ as $n \rightarrow + \infty$). \hfill $\square$

\subsection*{The second method}
Let $p$ be a prime number and $a$ be an integer non-multiple of $p$. For a given $n \in \N$, consider the rational function $R_n$ defined by:
$$
R_n(X) := \sum_{k = 1}^{n} \left(\frac{1}{a^k} + \frac{1}{(X - a)^k}\right) \frac{X^k}{k} = \sum_{k = 1}^{n} \dfrac{\left(\frac{X}{a}\right)^k}{k} + \sum_{k = 1}^{n} \dfrac{\left(\frac{X}{X - a}\right)^k}{k} .
$$
Consider also the real function $f$ defined at the neighborhood of $0$ by:
$$
f(X) := - \log\left(1 - X\right) ,
$$
which satisfies the functional equation:
\begin{equation}\label{eq13}
f\left(\frac{X}{a}\right) + f\left(\frac{X}{X - a}\right) = 0
\end{equation}
and whose the $n$\up{th} degree Taylor polynomial at $0$ is $\sum_{k = 1}^{n} \frac{X^k}{k}$. 

On the one hand, according to the well-known properties of Taylor polynomials, the $n$\up{th} degree Taylor polynomial of the function $X \stackrel{g}{\mapsto} f\left(\frac{X}{a}\right) + f\left(\frac{X}{X - a}\right)$ at $0$ is the same with the $n$\up{th} degree Taylor polynomial of
$$
\sum_{k = 1}^{n} \dfrac{\left(\frac{X}{a}\right)^k}{k} + \sum_{k = 1}^{n} \dfrac{\left(\frac{X}{X - a}\right)^k}{k} = R_n(X) .
$$
But on the other hand, in view of \eqref{eq13}, this $n$\up{th} degree Taylor polynomial of $g$ at $0$ is zero. Comparing these two results, we deduce that the multiplicity of $0$ in $R_n$ is at least $(n + 1)$. Consequently, $R_n(X)$ can be written as:
$$
R_n(X) = X^{n + 1} \cdot \frac{U_n(X)}{a^n (X - a)^n \lcm(1 , 2 , \dots , n)} ,
$$
where $U_n \in \Z[X]$. In particular, we have
$$
R_n(p) = p^{n + 1} \cdot \frac{U_n(p)}{a^n (p - a)^n \lcm(1 , 2 , \dots , n)} .
$$
Next, because $U_n(p) \in \Z$ (since $U_n \in \Z[X]$) and $a$ is not a multiple of $p$, then by taking the $p$-adic valuations in the two sides of the last identity, we derive that:
$$
\vartheta_p\left(R_n(p)\right) \geq n + 1 - \vartheta_p\left(\lcm(1 , 2 , \dots , n)\right) = n + 1 - \left\lfloor\log_p(n)\right\rfloor ,
$$
implying that the sequence $\left\{\vartheta_p\left(R_n(p)\right)\right\}_{n \geq 1}$ is unbounded from above, as required by Theorem \ref{t3}. \hfill $\square$

\begin{rmk}\label{rmk1}
Curiously, the two previous methods give the same upper bound
$$
\vartheta_p\left(\sum_{k = 1}^{n} \left(\frac{1}{a^k} + \frac{1}{(p - a)^k}\right) \frac{p^k}{k}\right) \geq n + 1 - \left\lfloor\log_p(n)\right\rfloor .
$$
Furthermore, this last estimate is remarkably very close to the optimal one of Theorem \ref{t1}.
\end{rmk}

In the third method below, we will show the non-boundness of the sequence in Theorem \ref{t3} without providing any estimate!

\subsection*{The third method}
Let $p$ be a prime number and $a$ be an integer non-multiple of $p$. For all $n \in \N$, set
$$
r_n := \left(\frac{1}{a^n} + \frac{1}{(p - a)^n}\right) \frac{p^n}{n} ~~\text{and}~~ s_n := \sum_{k = 1}^{n} r_k .
$$
The property we have to show is that the sequence ${\left(\vartheta_p(s_n)\right)}_{n \geq 1}$ is unbounded from above; in other words, we have that $\limsup_{n \rightarrow + \infty} \vartheta_p(s_n) = + \infty$. So, if we show the stronger property $\lim_{n \rightarrow + \infty} \vartheta_p(s_n) = + \infty$, then we are done. To do so, observe that:
\begin{align*}
\lim_{n \rightarrow + \infty} \vartheta_p(s_n) = + \infty & \Longleftrightarrow \lim_{n \rightarrow + \infty} {\left\vert s_n\right\vert}_p = 0 \\
& \Longleftrightarrow \lim_{n \rightarrow + \infty} s_n = 0 ~~~~~~~~~~ (\text{in the } p\text{-adic sense}) \\
& \Longleftrightarrow \sum_{k = 1}^{+ \infty} r_k = 0 ~~~~~~~~~~~~~ (\text{in the } p\text{-adic sense}) .
\end{align*}
Consequently, it suffices to show that:
\begin{equation}\label{eq14}
\sum_{k = 1}^{+ \infty} \left(\frac{1}{a^k} + \frac{1}{(p - a)^k}\right) \frac{p^k}{k} = 0
\end{equation}
(in the $p$-adic sense). Let us show \eqref{eq14}. By using the $p$-adic logarithm function (recalled in §\ref{sec1}), we have
\begin{align*}
\sum_{k = 1}^{+ \infty} \left(\frac{1}{a^k} + \frac{1}{(p - a)^k}\right) \frac{p^k}{k} & = \sum_{k = 1}^{+ \infty} \frac{\left(\frac{p}{a}\right)^k}{k} + \sum_{k = 1}^{+ \infty} \frac{\left(\frac{p}{p - a}\right)^k}{k} \\
& = - L_p\left(1 - \frac{p}{a}\right) - L_p\left(1 - \frac{p}{p - a}\right) \\[1mm]
& = - \left[L_p\left(\frac{a - p}{a}\right) + L_p\left(\frac{- a}{p - a}\right)\right] \\[1mm]
& = - L_p\left(\frac{a - p}{a} \cdot \frac{- a}{p - a}\right) \\[1mm]
& = - L_p(1) = 0 ,
\end{align*}
as required. The non-boundness from above of the sequence ${\left(\vartheta_p(s_n)\right)}_{n \geq 1}$ follows. \hfill $\square$

\rhead{\textcolor{OrangeRed3}{\it References}}

\end{document}